\documentclass[10pt]{article}

\input xy
\xyoption{all}
\usepackage{amsmath,hyperref,amssymb,amsthm,mathrsfs}

\newtheorem{thm}{Theorem}[section]
\newtheorem{cor}[thm]{Corollary}
\newtheorem{lem}[thm]{Lemma}
\newtheorem{prop}[thm]{Proposition}

\theoremstyle{remark}

\def\Stein{{\mathrm{e}}}

\def\Im{\mathrm{Im}}
\def\reg{\widetilde{\rho}}

\def\D{\mathbb{D}}

\def\Ext{\mathrm{Ext}}
\def\Hom{\mathrm{Hom}}

\def\F{\mathbb F}

\def\Z{\mathbb Z}
\def\N{\mathbb N}
\def\A{\mathscr A}
\def\U{\mathscr U}

\def\T{\mathrm{T}}

\begin{document}

\title{Takayasu cofibrations revisited}
\author{Nguyen D.H. Hai and Lionel Schwartz}



\maketitle

%


\section{Introduction}

Given a natural number $n$, let $\reg_n$ be the reduced real regular representation of the elementary abelian $2$-group $V_n:=(\Z/2)^n$. Let $BV_n^{k\reg_n}$, $k\in \N$, denote the Thom space over the classifying space $BV_n$ associated to the direct sum of $k$ copies of the representation $\reg_n$. Following S. Takayasu \cite{Tak99}, let $M(n)_k$ denote the stable summand of $BV_n^{k\reg_n}$ which corresponds to the Steinberg module of the general linear group $GL_n(\F_2)$ \cite{MP83}.

Takayasu constructed in \cite{Tak99} a cofibration of the following form: 
$$\Sigma^kM(n-1)_{2k+1}\to M(n)_k\to M(n)_{k+1}.$$
This generalised the splitting of Mitchell and Priddy $M(n)\simeq L(n)\vee L(n-1)$, 
where $M(n)=M(n)_0$ and $L(n)=M(n)_1$ \cite{MP83}. Takayasu also considered the spectra $M(n)_k$ associated to the virtual representations $k\reg_n$, $k<0$, and proved that the above cofibrations are still valid for these spectra. Here and below, all spectra are implicitly completed at the prime two. 

Note that the spectra $M(n)_k$, $k\ge 0$, are used in the description of layers of the Goodwillie tower of the identity functor evaluated at spheres \cite{AM99, AD01}, and the above cofibrations can also be deduced by combining Goodwillie calculus with the James fibration,  as described by M. Behrens in \cite[Chapter 2]{Behrens}.

The purpose of this note is to give another proof for the existence of the above cofibrations for the cases $k\in \N$. 
This will be carried out by employing techniques in the category of unstable modules over the Steenrod algebra \cite{Sch94}. Especially, the action of Lannes' $\T$-functor on the Steinberg unstable modules (see \S \ref{vanish}), will play a crucial role in studying the vanishing of some extension groups of modules over the Steenrod algebra.

\section{Algebraic short exact sequences}
In this section, we recall the linear structure of the mod $2$ cohomology of $M(n,k)$ and the short exact sequences relating these $\A$-modules.
 Recall that the general linear group $GL_n:=GL_n(\F_2)$ acts on $H^*V_n\cong\F_2[x_1,\ldots,x_n]$ by the rule:
$$(gF)(x_1,\cdots,x_n):=F(\sum_{i=1}^ng_{i,1}x_i,\cdots,\sum_{i=1}^ng_{i,n}x_i),$$
where $g=(g_{i,j})\in GL_n$ and $F(x_1,\cdots,x_n)\in \F_2[x_1,\ldots,x_n]$. This action commutes with the action of the Steenrod action on $\F_2[x_1,\ldots,x_n]$.

By definition, the Thom class of the vector bundle associated to the reduced regular representation $\reg_n$ is given by the top Dickson invariant:
$$\omega_n=\omega_n(x_1,\ldots,x_n):=\prod_{0\not =x\in \F_2\langle x_1,\ldots,x_n\rangle}x.$$
Recall also that the Steinberg idempotent $\Stein_n$ of $\F_2[GL_n]$ is given by 
$$\Stein_n:=\sum_{b\in B, \sigma\in \Sigma_n}b\sigma,$$
where $B_n$ is the subgroup of upper triangular matrices in 
$GL_n$ and $\Sigma_n$ the subgroup of permutation matrices. 

Let $M_{n,k}$ denote the mod $2$ cohomology of the spectrum $M(n,k)$. By Thom isomorphism, we have an isomorphism of $\A$-modules:
$$M_{n,k}\cong \Im [\omega_n^k H^*BV_n\xrightarrow{\Stein_n}\omega_n^k H^*BV_n].$$
We note that $M_{n,k}$ is invariant under the action of the group $B_n$. 

\begin{prop}[\cite{HaiJA13}]
A basis for the graded vector space $M_{n,k}$ is given by 
$$\{\Stein_n(\omega_1^{i_1-2i_2}\cdots \omega_{n-1}^{i_{n-1}-2i_n}\omega_n^{i_n}) \mid \text{$i_j> 2i_{j+1}$ for $1\le j\le n-1$ and $i_n\ge k$}\}.$$
\end{prop}

\begin{thm}[cf. \cite{Tak99}]
Let $\alpha: M_{n,k+1}\rightarrow M_{n,k}$ be the natural inclusion and let $\beta : M_{n,k}\rightarrow\Sigma^k M_{n-1,2k+1}$ be the map given by 
$$\beta(\omega_{i_1,\cdots,i_{n}})=
\begin{cases}
0, & i_{n}>k,\\
\Sigma^{k}\omega_{i_1,\cdots,i_{n-1}}, & i_{n}=k.
\end{cases}
$$
Then 
$$0\to M_{n,k+1}\xrightarrow{\alpha} M_{n,k}\xrightarrow{\beta} \Sigma^k M_{n-1,2k+1}\to 0$$
is a short exact sequence of $\A$-modules:
\end{thm}

The exactness of the sequence can be proved by using the following:
\begin{lem}[{\cite[Proposition 1.2]{NST09b}}] We have
$$\omega_{i_1,\cdots,i_n}=\omega_{i_1,\cdots,i_{n-1}}x_n^{i_n}+ \text{terms 
$\omega_{j_1,\cdots,j_{n-1}}x_n^{j}$ with $j>j_n$.}$$
\end{lem}

Note also that a minimal generating set for the $\A$-module $M_{n,k}$ was constructed in \cite{HaiJA13}, generating the work of Inoue \cite{Ino02}.

\section{Existence of the cofibrations}

A spectrum $X$ is said to be of finite type if its mod $2$ cohomology, $H^*X$, is finite-dimensional in each degree. Recall that given a sequence $X\to Y\to Z$ of spectra of finite type, if the composite $X\to Z$ is homotopically trivial and the induced sequence $0\to H^*Z\to H^*Y\to H^*X\to 0$ is a short exact sequence of $\A$-modules , then $X\to Y\to Z$ is a cofibration.

We wish to realise the algebraic short sequence $$0\to M_{n,k+1}\xrightarrow{\alpha} M_{n,k}\xrightarrow{\beta} \Sigma^k M_{n-1,2k+1}\to 0.$$
 by a cofibration of spectra
 $$\Sigma^kM(n-1)_{2k+1}\to M(n)_k\to M(n)_{k+1}.$$

The inclusion of $k\reg_n$ into $(k+1)\reg_n$ induces a natural map of spectra
$$i:M(n)_k\to M(n)_{k+1}.$$
It is clear that this map realises the inclusion of $\A$-modules $\alpha:M_{n,k+1}\rightarrow M_{n,k}$. 
We wish now to realise the $\A$-linear map $\beta:M_{n,k}\rightarrow \Sigma^k M_{n-1,2k+1}$ by a map of spectra
$$j:\Sigma^k M(n-1)_{2k+1}\to M(n)_k$$
such that the composite $i\circ j$ is homotopically trivial. 
The existence of such a map is an immediate consequence of the following result.
\begin{thm}\label{main}
For all $k\ge 0$, we have
\begin{enumerate}
\item The natural map 
$[\Sigma^k M(n-1)_{2k+1}, M(n)_k]\to \Hom_\A(M_{n,k}, \Sigma^k M_{n-1,2k+1})$
is onto. 
\item The group $[\Sigma^{k}M(n-1)_{2k+1},M(n)_{k+1}]$ is trivial.
\end{enumerate}
\end{thm}
 
The theorem is proved by using the Adams spectral sequence 
$$\Ext_\A^{s}(H^*Y,\Sigma^t H^*X)\Longrightarrow [\Sigma^{t-s}X,Y].$$
For the first part, it suffices to prove that
\begin{equation}
\Ext^s_\A(M_{n,k}, \Sigma^{k+t} M_{n-1,2k+1})=0\quad  \text{for $s\ge 0$ and $t-s<0$,} 
\end{equation} 
so that the non-trivial elements in $\Hom_\A(M_{n,k}, \Sigma^k M_{n-1,2k+1})$ are permanent cycles. For the second part, it suffices to prove that 
\begin{equation}
\Ext^s_\A(M_{n,k+1}, \Sigma^{k+s} M_{n-1,2k+1})=0, \quad  \text{for $s\ge 0$}.
\end{equation}
Here and below, $\A$-linear maps are of degree zero, and so $\Ext_\A^{s}(M,\Sigma^tN)$ is the same as the group denoted by $\Ext_\A^{s,t}(M,N)$ in the traditional notation.

The vanishing of the above extension groups will be proved in the next section.

\section{On the vanishing of $\Ext_\A^{s}(M_{n,k},\Sigma^{i+s} M_{m,j})$} \label{vanish}

In this section, we establish a sufficient condition for the vanishing of the extension groups $\Ext_\A^{s}(M_{n,k},\Sigma^{i+s} M_{m,j})$. Note that we always consider the modules $M_{n,k}$ with $k\ge 0$.

Below we consider seperately two cases for the vanishing of the groups $\Ext_\A^{s}(M_{n,k},\Sigma^{i+s} M_{m,j})$:
Proposition \ref{base} gives a condition for the case $j=0$ and Proposition \ref{base2} gives a condition for the case $j>0$.

\begin{prop}\label{base}
Suppose $n>m\ge 0$ and $-\infty<i<|M_{n-m,k}|$. Then
$$\Ext_\A^s(M_{n,k},\Sigma^{i+s}M_{m})=0,\quad s\ge 0.$$
\end{prop}

Here $|M|$ denotes the connectivity of $M$, i.e. the minimal degree in which $M$ is non-trivial. 

To consider the case $j>0$, put $\varphi(j)=2j-1$ and $$F(i,j,q)=i+j+\varphi(j)+\varphi^2(j)+\cdots +\varphi^{q-1}(j),$$ where  $\varphi^t$ is the $t$-fold composition of $\varphi$. Explicitly, 
$$F(i,j,q)=i+(j-1)(2^q-1)+q.$$ 
Note that 
$F(i+j,2j-1,q)=F(i,j,q+1)$ and $F(i,j',q)\le F(i,j,q)$ if $j'\le j$.

\begin{prop}\label{base2}
Suppose $n>m\ge 0$, $j>0$ and $F(i,j,q)<|M_{n-m+q,k}|$ for $0\le q\le m$. Then
$$\Ext_\A^{s}(M_{n,k},\Sigma^{i+s} M_{m,j})=0,\quad s\ge 0.$$
\end{prop}

Recall that Lannes' $\T$-functor is left adjoint to 
the tensoring with $H:=H^*B\Z/2$ in the category $\U$ of unstable modules over the Steenrod algebra \cite{Lan92}.
We need the following result, observed by Harris and Shank \cite{Harris-Shank-92}, to prove Proposition \ref{base}. 

\begin{prop}[{Carlisle-Kuhn \cite[6.1]{CK96} combined with Harris-Shank \cite[4.19]{Harris-Shank-92}}] There is an isomorphism of unstable modules
$$\T(L_n)\cong L_n\oplus (H\otimes L_{n-1}).$$
Here $L_n=M_{n,1}$. 
\end{prop}

\begin{cor}\label{connex}For $n\ge m$, we have
$|\T^m(M_{n,k})| = |M_{n-m,k}|$.
\end{cor}

\begin{proof}
By iterating the action of $\T$ on $L_n$, we see that there is an isomorphism of unstable modules
\begin{eqnarray*}
\T^m(L_n)\cong \bigoplus_{i=0}^m [H^{\otimes i}\otimes L_{n-i}]^{\oplus a_{i}},
\end{eqnarray*}
where $a_{i}$ are certain positive integers depending only on $m$.
By using the exactitude of $\T^m$ and the short exact sequences
$$0\to M_{n,k+1}\xrightarrow{\alpha} M_{n,k}\xrightarrow{\beta} \Sigma^k M_{n-1,2k+1}\to 0,$$ it is easy to prove by induction that there is an isomorphism of graded vector spaces
$$\T^m(M_{n,k})\cong \bigoplus_{i=0}^m [H^{\otimes i}\otimes M_{n-i,k}]^{\oplus a_{i}}.$$
The corollary follows.
\end{proof}

\begin{proof}[Proof of Proposition \ref{base}]
Fix $i$, $s$ and take a positive integer $q$ big enough such that 
$i+s+q$ is positive. We have $$\Ext_\A^s(M_{n,k},\Sigma^{i+s}M_{m})=\Ext_\A^s(\Sigma^q M_{n,k},\Sigma^{i+s+q}M_{m}).$$
Using the Grothendieck spectral sequence, we need to prove that
$$\Ext^{s-j}_\U(\D_{j}\Sigma^q M_{n,k}, \Sigma^{i+s+q}M_{m}) =0,\quad 0\le j\le s.$$
Here $\D_j$ is the $j$th-derived functor of the destabilisation functor $$\D:\A\text{-mod} \to \U$$ from the category of $\A$-modules to the category of unstable $\A$-modules \cite{LZ87}. 

As $M_n$ is $\U$-injective, it is easily seen that $\Sigma^{\ell}M_{m}$ has a $\U$-injective resolution $I^\bullet$ where $I^t$ is a direct sum of $M_{m}\otimes J(a)$ with $a\le \ell-t$, where $J(a)$ is the Brown-Gitler module \cite{LZ86}. So we need to prove that, for  $a\le (i+s+q)- (s-j)=i+j+q$,  we have 
$$\Hom_\U (\D_j\Sigma^q M_{n,k}, M_{m}\otimes J(a))=0.$$ 
By Lannes-Zarati \cite{LZ87}, we have $$\D_j\Sigma^q M_{n,k}=\Sigma R_j\Sigma^{j-1+q} M_{n,k}\subset \Sigma^{j+q}H^{\otimes j}\otimes M_{n,k},$$ 
where $R_j$ is the Singer functor. 
It follows that 
$\Hom_\U (\D_j\Sigma^q M_{n,k}, M_{m}\otimes J(a))$ is a quotient of 
$$\Hom_\U (\Sigma^{j+q}H^{\otimes j}\otimes M_{n,k}, M_{m}\otimes J(a))$$
which is in turn a subgroup of 
$$\Hom_\U (\Sigma^{j+q}H^{\otimes j}\otimes M_{n,k}, H^{\otimes m}\otimes J(a))={\big((\T^m(\Sigma^{j+q}H^{\otimes j}\otimes M_{n,k}))^a\big)}^*.$$
This group is trivial because, by Corollary \ref{connex}, we have
 $$|\T^m(\Sigma^{j+q}H_i\otimes M_{n,k})|=|\Sigma^{j+q} M_{n-m,k}|= |M_{n-m,k}|+j+q>i+j+q\ge a.$$ 
The proposition follows.
\end{proof}

\begin{proof}[Proof of Propositiont \ref{base2}]
We prove the proposition by induction on $m\ge 0$. By noting that $M_{0,j}=\Z/2$, the case $m=0$ is a special case of Proposition \ref{base}.

Suppose $m> 0$. 
For simplicity, put 
$E^s(\Sigma^iM_{m,j})=\Ext_\A^{s}(M_{n,k},\Sigma^{i+s} M_{m,j}).$ 
The short exact sequence of $\A$-modules
$M_{m,j}\hookrightarrow M_{m,j-1} \twoheadrightarrow \Sigma^{j-1}M_{m-1,2j-1}$
induces a long exact sequence in cohomology
$$\cdots \to E^{s-1}(\Sigma^{i+j}M_{m-1,2j-1}) \to E^s(\Sigma^i M_{m,j})\to E^s(\Sigma^i M_{m,j-1}) \to \cdots $$
So from the cofiltration of $M_{m,j}$
$$\xymatrix{
\Sigma^iM_{m,j}\ar@{^(->}[r] & \Sigma^i M_{m,j-1}\ar@{^(->}[r]\ar@{->>}[d]&  \cdots \ar@{^(->}[r] &  \Sigma^i M_{m,1}\ar@{^(->}[r]\ar@{->>}[d] &\Sigma^i M_{m}\ar@{->>}[d]\\
 & \Sigma^{i+j-1} M_{m-1,2j-1} 
& & \Sigma^{i+1} M_{m-1,3} & \Sigma^iM_{m-1,1}
}$$
we see that, in order to prove $E^s(\Sigma^i M_{m,j})=0$, it suffices to prove that the groups  $E^{s-1}(\Sigma^{i+j'}M_{m-1,2j'-1})$, $1\le j'\le j$, and $E^s(\Sigma^i M_{m})$, are trivial.

By Proposition \ref{base}, $E^s(\Sigma^i M_{m})$ is trivial since $i=F(i,j,0)<|M_{n,k}|$.
For $1\le j'\le j$ and $0\le q\le m-1$, we have 
$$F(i+j',2j'-1,q)=F(i,j',q+1)\le F(i,j,q+1)< |M_{n-m+1+q,k}|.$$
By inductive hypothesis for $m-1$, we have 
$E^{s-1}(\Sigma^{i+j'}M_{m-1,2j'-1})=0.$ The proposition is proved.
\end{proof}

We are now ready to prove Theorem \ref{main}.
Recall that the connectivity of $M_{n,k}$ 
is given by $$|M_{n,k}|=1+3+\cdots +(2^{n-1}-1)+(2^n-1)k.$$  

\begin{proof}[Proof of Theorem \ref{main} (1)]
Using the Adams spectral sequence, it suffices to prove that$$\Ext^s_\A(M_{n,k}, \Sigma^{k+t} M_{n-1,2k+1})=0\quad  \text{for $s\ge 0$ and $t-s<0$.}$$
For $q\ge 0$, we have 
$$F(k+t-s,2k+1,q)=k+t-s +2k(2^q-1) +q < (2^{q+1}-1)k +q
\le |M_{q+1,k}|.$$
The vanishing of the extension groups follows from
Proposition \ref{base2}.
\end{proof}

\begin{proof}[Proof of Theorem \ref{main} (2)] Using the Adams spectral sequence, it suffices to prove that $$\Ext^s_\A(M_{n,k+1}, \Sigma^{k+s} M_{n-1,2k+1})=0, \quad  \text{for $s\ge 0$}.$$
For $q\ge 0$, we have 
$$F(k,2k+1,q)=k+2k(2^q-1)+t=(2^{q+1}-1)k+q<|M_{q+1,k+1}|.$$
The vanishing of the extension groups follows from Proposition \ref{base2}.
\end{proof}


\begin{thebibliography}{10}

\bibitem{AD01}
G.~Z. Arone and W.~G. Dwyer.
\newblock Partition complexes, {T}its buildings and symmetric products.
\newblock {\em Proc. London Math. Soc. (3)}, 82(1):229--256, 2001.

\bibitem{AM99}
Greg Arone and Mark Mahowald.
\newblock The {G}oodwillie tower of the identity functor and the unstable
  periodic homotopy of spheres.
\newblock {\em Invent. Math.}, 135(3):743--788, 1999.

\bibitem{Behrens}
Mark Behrens.
\newblock The {G}oodwillie tower and the {EHP} sequence.
\newblock {\em Mem. Amer. Math. Soc.}, 218(1026):xii+90, 2012.

\bibitem{CK96}
D.~P. Carlisle and N.~J. Kuhn.
\newblock Smash products of summands of {$B({\bf Z}/p)\sp n\,\sb +$}.
\newblock In {\em Algebraic topology ({E}vanston, {IL}, 1988)}, volume~96 of
  {\em Contemp. Math.}, pages 87--102. Amer. Math. Soc., Providence, RI, 1989.

\bibitem{HaiJA13}
Nguyen Dang~Ho Hai.
\newblock Generators for the mod 2 cohomology of the {S}teinberg summand of
  {T}hom spectra over {$B(\Bbb{Z}/2)\sp n$}.
\newblock {\em J. Algebra}, 381:164--175, 2013.

\bibitem{NST09b}
Nguyen Dang~Ho Hai, Lionel Schwartz, and Tran~Ngoc Nam.
\newblock La fonction g\'en\'eratrice de {M}inc et une ``conjecture de
  {S}egal'' pour certains spectres de {T}hom.
\newblock {\em {A}dv. in {M}ath.}, 225(3):1431--1460, 2010.

\bibitem{Harris-Shank-92}
John~C. Harris and R.~James Shank.
\newblock Lannes' {$T$} functor on summands of {$H\sp *(B({\bf Z}/p)\sp s)$}.
\newblock {\em Trans. Amer. Math. Soc.}, 333(2):579--606, 1992.

\bibitem{Ino02}
Masateru Inoue.
\newblock {$\mathcal A$}-generators of the cohomology of the {S}teinberg
  summand {$M(n)$}.
\newblock In {\em Recent progress in homotopy theory ({B}altimore, {MD},
  2000)}, volume 293 of {\em Contemp. Math.}, pages 125--139. Amer. Math. Soc.,
  Providence, RI, 2002.

\bibitem{Lan92}
Jean Lannes.
\newblock Sur les espaces fonctionnels dont la source est le classifiant d'un
  {$p$}-groupe ab\'elien \'el\'ementaire.
\newblock {\em Inst. Hautes \'Etudes Sci. Publ. Math.}, (75):135--244, 1992.
\newblock With an appendix by Michel Zisman.

\bibitem{LZ86}
Jean Lannes and Sa{\"{\i}}d Zarati.
\newblock Sur les {${\mathcal U}$}-injectifs.
\newblock {\em Ann. Sci. \'Ecole Norm. Sup. (4)}, 19(2):303--333, 1986.

\bibitem{LZ87}
Jean Lannes and Sa{\"{\i}}d Zarati.
\newblock Sur les foncteurs d\'eriv\'es de la d\'estabilisation.
\newblock {\em Math. Z.}, 194(1):25--59, 1987.

\bibitem{MP83}
Stephen~A. Mitchell and Stewart~B. Priddy.
\newblock Stable splittings derived from the {S}teinberg module.
\newblock {\em Topology}, 22(3):285--298, 1983.

\bibitem{Sch94}
Lionel Schwartz.
\newblock {\em Unstable modules over the {S}teenrod algebra and {S}ullivan's
  fixed point set conjecture}.
\newblock Chicago Lectures in Mathematics. University of Chicago Press,
  Chicago, IL, 1994.

\bibitem{Tak99}
Shin-ichiro Takayasu.
\newblock On stable summands of {T}hom spectra of {$B({\bf Z}/2)\sp n$}
  associated to {S}teinberg modules.
\newblock {\em J. Math. Kyoto Univ.}, 39(2):377--398, 1999.

\end{thebibliography}
\end{document}